\documentclass[11pt,twoside]{amsart}
\usepackage{amsmath}
\usepackage{amsthm}
\usepackage{amsfonts}
\usepackage{amssymb}
\usepackage{latexsym}
\usepackage{mathrsfs}
\usepackage{amsmath}
\usepackage{amsthm}
\usepackage{amsfonts}
\usepackage{amssymb}
\usepackage{latexsym}
\usepackage{geometry}
\usepackage{dsfont}
\usepackage[dvips]{graphicx}
\usepackage{color}
\usepackage[all]{xy}
\usepackage[pagebackref]{hyperref}
\renewcommand*{\backref}[1]{}\renewcommand*{\backrefalt}[4]{\ifcase #1 (\tt not cited)\or (\tt cited on page~#2)\else (\tt cited on pages~#2)\fi}

\date{}
\allowdisplaybreaks[4] \footskip=15pt
\renewcommand{\uppercasenonmath}[1]{}

\usepackage{graphicx,amssymb}
\usepackage[all]{xy}
\usepackage{amsmath}

\allowdisplaybreaks
\usepackage{amsthm}
\usepackage{color}

\theoremstyle{plain}
\newtheorem{theorem}{Theorem}[section]
\newtheorem{proposition}[theorem]{Proposition}
\newtheorem{lemma}[theorem]{Lemma}
\newtheorem{corollary}[theorem]{Corollary}
\theoremstyle{definition}
\newtheorem{example}[theorem]{Example}
\newtheorem{definition}[theorem]{Definition}

\theoremstyle{definition}

\theoremstyle{remark}
\newtheorem{remark}[theorem]{Remark}

\newcommand{\pf}{\noindent\begin {proof}}
\newcommand{\epf}{\end{proof}}

\newcommand{\Ker}{\mbox{\rm Ker}}
\newcommand{\Ext}{\mbox{\rm Ext}}
\newcommand{\Hom}{\mbox{\rm Hom}}
\newcommand{\Tor}{\mbox{\rm Tor}}

\newcommand{\Id}{\mathrm{Id}}

\def\m{{\frak m}}

\def\GV{{\rm GV}}
\def\tor{{\rm tor_{\rm GV}}}
\def\Hom{{\rm Hom}}
\def\Ext{{\rm Ext}}
\def\Tor{{\rm Tor}}

\def\ker{{\rm ker}}
\def\Ker{{\rm Ker}}

\def\GV{{\rm GV}}
\def\Max{{\rm Max}}

\def\fkm{{\mathfrak{m}}}

\begin{document}

\title[Characterizations of Pr\"ufer $v$-Multiplication Domains]{ Module-Theoretic Characterizations of Pr\"ufer $v$-Multiplication Domains}

\author [Zhang] {Xiaolei Zhang}

\author [Kim] {Hwankoo Kim$^\dag$}

\date{\today}
\thanks{$\dag$: Corresponding author}

\maketitle

\begin{abstract}
	We present unified $w$-theoretic characterizations of Pr\"ufer $v$-multiplication domains (P$v$MDs). 
	A module-theoretic perspective shows that torsion submodules are $w$-pure, and for $(w$-)$\,$finitely generated modules $M$, the canonical sequence 
	$0\to T(M)\to M\to M/T(M)\to 0$ $w$-splits, resolving an open question of Geroldinger--Kim--Loper. 
	In a $w$-version of Hattori-Davis theory, these conditions are equivalent to $\Tor^R_2(M,N)$ being $\GV$-torsion for all $R$-modules $M,N$, equivalently $w$-w.gl.dim$(R)\le 1$, or $\Tor^R_1(X,A)$ being $\GV$-torsion for all $X$ and torsion-free $A$, or the Davis map $A\otimes_R B \to \mathcal T\otimes_K \mathcal S$ having $\GV$-torsion kernel. 
	From an overring viewpoint, $R$ is a P$v$MD if and only if for every $R\subseteq T\subseteq K$ and every $w$-maximal ideal $\fkm$, the localization $R_{\fkm}\to T_{\fkm}$ is a flat epimorphism, so that each overring is $w$-flat and the inclusion is $w$-epimorphic. 
	Finally, $R$ is a P$v$MD if and only if every pure $w$-injective divisible $R$-module is injective.
\end{abstract}

\section{introduction}

Throughout, $R$ denotes an integral domain with $1\neq 0$, and all modules are unitary.

Pr\"ufer introduced in 1932 those domains in which every nonzero finitely generated ideal is invertible; Krull later named them \emph{Pr\"ufer domains} \cite{P32,K37}. Since then a rich module-theoretic theory has developed (see \cite{FHP97}). In the last decades, following Gilmer’s star-operation framework \cite{G72}, many classical properties were transported from Pr\"ufer domains to their $v$/$t$/$w$ analogues. In particular, Pr\"ufer $v$-multiplication domains (P$v$MDs), originating with Griffin \cite{G67}, are those domains in which every nonzero finitely generated ideal is $v$-invertible (equivalently $t$- or $w$-invertible); they play the role of the $v$- (or $t$-, $w$-) version of Pr\"ufer domains. See, e.g., \cite{wq15,xw18,XW18} for module-theoretic characterizations in this milieu.

\smallskip
\noindent\textbf{Aim and main results.}  
This paper develops a coherent $w$-theoretic package of criteria that are \emph{equivalent} to the P$v$MD property, drawn from four classical directions:

\begin{itemize}
	\item \emph{Torsion-submodule criteria (resolving an open question):}  
	For any $R$-module $M$, the torsion submodule $T(M)$ is $w$-pure, and for every $(w$-)$\,\!$finitely generated module $M$, the canonical exact sequence
	\[
	0 \longrightarrow T(M) \longrightarrow M \longrightarrow M/T(M) \longrightarrow 0
	\]
	$w$-splits. This settles an open question posed by Geroldinger--Kim--Loper (\cite{ghl25}).
	
	\item \emph{$w$-Hattori-Davis $($homological/tensor$) $ criteria:}  
	Theorem~\ref{thm:P$v$MD-Hattori-Davis} shows that $R$ is a P$v$MD if and only if 
	$\Tor^R_2(M,N)$ is $\GV$-torsion for all $R$-modules $M,N$ 
	(equivalently, $w$-w.gl.dim$(R)\le 1$), 
	if and only if $\Tor^R_1(X,A)$ is $\GV$-torsion for all $R$-modules $X$ and torsion-free $A$, 
	and if and only if the Davis map has $\GV$-torsion kernel for all $R$-subalgebras $A,B$ of $K$.
	
	\item \emph{$w$-local flat-epimorphism criterion for overrings:}  
	Theorem~\ref{thm:pvmd-flat-epi} proves that $R$ is a P$v$MD if and only if, 
	for every overring $R\subseteq T\subseteq K$ and every $\fkm\in w$-$\Max(R)$, 
	the localized inclusion $R_{\fkm}\to T_{\fkm}$ is a flat epimorphism.  
	In particular, each overring is $w$-flat and the inclusion $R\hookrightarrow T$ is a $w$-epimorphism.
	
	\item \emph{Pure $w$-injective module characterization:}  
	Theorem~\ref{thm:pvmd-winj} shows that $R$ is a P$v$MD if and only if every pure $w$-injective divisible $R$-module is injective.   
	This provides a module-theoretic perspective on P$v$MDs, linking them to Fuchs’ classical result on injectivity of divisible modules.
\end{itemize}

\smallskip
\noindent\textbf{Notation and preliminaries.}
We briefly recall the $w$-setup used throughout (details and proofs can be found in \cite{kw14,wq20,WK15,wk24,xw18}).  
A finitely generated ideal $J\subseteq R$ is a \emph{\GV-ideal} (write $J\in \GV(R)$) if the natural map $R\!\to\!\Hom_R(J,R)$ is an isomorphism.  
For an $R$-module $M$, set
\[
\tor(M):=\{\,x\in M\mid Jx=0\ \text{for some }J\in\GV(R)\,\}.
\]
Then $M$ is \emph{\GV-torsion} if $\tor(M)=M$ and \emph{\GV-torsion-free} if $\tor(M)=0$.  
The \emph{$w$-envelope} of a \GV-torsion-free $M$ is
\[
M_w:=\{\,x\in E(M)\mid Jx\subseteq M\ \text{for some }J\in\GV(R)\,\},
\]
with $E(M)$ the injective envelope; a \GV-torsion-free $M$ is a \emph{$w$-module} if $M_w=M$.  
Maximal $w$-ideals (maximal $w$-submodules of $R$) are prime; their set is $w$-$\Max(R)$ \cite[Thm.~6.2.14]{wk24}.  
A sequence is \emph{$w$-exact} if it becomes exact after localizing at every $\fkm\in w$-$\Max(R)$.  
An $R$-module $M$ is \emph{$w$-finitely generated} if it receives a $w$-epimorphism from a finitely generated free module.  
$M$ is \emph{$w$-flat} iff $\Tor^R_1(N,M)$ is \GV-torsion for all $N$ \cite[Thm.~3.3]{kw14}; \emph{$w$-projective} iff $\Ext^1_R(L(M),N)$ is \GV-torsion for all torsion-free $w$-modules $N$, where $L(M)=(M/\tor(M))_w$ \cite{WK15}.  
A short exact sequence is \emph{$w$-split} if it splits after multiplying by some $J\in\GV(R)$ in the sense of \cite{wq20}; $w$-split $\Rightarrow$ $w$-pure \cite{xw18}, and $w$-split $\Rightarrow$ $w$-projective $\Rightarrow$ $w$-flat (the converses generally fail).

These conventions will be in force for the remainder of the paper.

\section{Characterizing P$v$MDs via $w$-Purity and $w$-Projectivity}

In this section, we provide a homological characterization of P$v$MDs in terms of torsion submodules and $w$-projectivity. 
Our main theorem gives several equivalent conditions for an integral domain to be a P$v$MD, 
thus resolving an open question posed by Geroldinger, Kim, and Loper \cite{ghl25}. 
The characterization connects module-theoretic properties, such as $w$-purity and $w$-split exact sequences, 
with the ideal-theoretic structure of the domain, offering a new perspective on P$v$MDs.

Following Wang-Kim \cite{WK15}, the \emph{$w$-Nagata ring} of $R$ is
\[
R\{x\}\ :=\ R[X]_{S_w},\qquad
S_w\ :=\ \{\,f\in R[X]\mid c(f)\in \GV(R)\,\},
\]
where $c(f)$ denotes the content ideal of $f$.  
Similarly, for an $R$-module $M$, the \emph{$w$-Nagata module} is
\[
M\{x\}\ :=\ M[X]_{S_w}\ \cong\ M\otimes_R R\{x\}.
\]
The next lemma records the basic transfer between the $w$-setting and the classical setting via the functor $(-)\{x\}$.

\begin{lemma}\label{Nag}
	Let $R$ be a ring. Then:
	\begin{enumerate}
		\item[\textup{(1)}] \textup{\cite[Theorem 3.9]{WK15}} An $R$-module $M$ is $w$-finitely generated if and only if $M\{x\}$ is finitely generated over $R\{x\}$.
		\item[\textup{(2)}] \textup{\cite[Lemma 4.1]{wq15}} An $R$-module $M$ is $w$-flat if and only if $M\{x\}$ is flat over $R\{x\}$.
		\item[\textup{(3)}] \textup{\cite[Theorem 3.11]{WK15}} An $R$-module $M$ is $w$-finitely generated and $w$-projective if and only if $M\{x\}$ is finitely generated projective over $R\{x\}$.
	\end{enumerate}
\end{lemma}

Recall \cite{EJ00} that an injective $R$-module $E$ is an \emph{injective cogenerator} if for every nonzero $R$-module $M$ there exists a nonzero homomorphism $f:M\to E$.
For the $w$-Nagata ring $R\{x\}$, one has
\[
\{\ \fkm\{x\}\mid \fkm\in w\text{-}\Max(R)\ \}
\]
as the set of all maximal ideals of $R\{x\}$ \cite[Proposition 6.6.16(2)]{wk24}. Set
\[
E'\ :=\ \prod_{\fkm\in w\text{-}\Max(R)} E_R\!\big(R\{x\}/\fkm\{x\}\big),
\]
where $E_R(-)$ denotes the injective envelope regarded as an $R$-module. Since each $R\{x\}/\fkm\{x\}$ is a $w$-module over $R$, the product $E'$ is an injective $w$-module over $R$. Define
\[
\widetilde E\ :=\ \Hom_R\!\big(R\{x\},E'\big).
\]
Then $\widetilde E$ is naturally an $R\{x\}$-module and, by \cite[Lemma 3.10]{zw21}, it is an injective cogenerator of the category of $R\{x\}$-modules.

\begin{lemma}\label{cog-tor}
	\textup{\cite[Corollary 3.11]{zw21}}
	Let $M$ be an $R$-module. The following are equivalent:
	\begin{enumerate}
		\item[\textup{(i)}] $M$ is $\GV$-torsion;
		\item[\textup{(ii)}] $\Hom_R(M,E)=0$ for every injective $w$-module $E$;
		\item[\textup{(iii)}] $\Hom_R(M,\widetilde E)=0$.
	\end{enumerate}
\end{lemma}

\begin{proposition}\label{wf-wp}
	Let $R$ be an integral domain. Then every $w$-finitely generated $w$-flat $R$-module is $w$-projective.
\end{proposition}

\begin{proof}
	Let $M$ be a $w$-finitely generated $w$-flat $R$-module. By Lemma~\ref{Nag}\,(1)-(2), $M\{x\}$ is a finitely generated flat $R\{x\}$-module (see also \cite[Theorem 6.6.4]{wk24}). Since $R\{x\}$ is an integral domain, $M\{x\}$ is finitely generated projective over $R\{x\}$ \cite[Lemma 3]{PR04}. Hence, by Lemma~\ref{Nag}\,(3) (or \cite[Theorem 6.7.29]{wk24}), $M$ is $w$-projective.
\end{proof}

\begin{proposition}\label{w-sp-w-pure}
	Let $R$ be a commutative ring $($not necessarily an integral domain$)$. Then the following statements hold:
	\begin{enumerate}
		\item Every $w$-split short exact sequence is $w$-pure.
		\item The direct limit of $w$-pure short exact sequences is also $w$-pure.
	\end{enumerate}	
\end{proposition}

\begin{proof}
	(1) Let 
	\[
	\zeta:\ 0 \longrightarrow A \xrightarrow{f} B \xrightarrow{g} C \longrightarrow 0
	\]
	be a $w$-split short exact sequence. Then there exist $J=\langle d_1,\dots,d_n\rangle\in \GV(R)$ and $h_1,\dots,h_n\in\Hom_R(C,B)$ such that 
	\[
	d_k\cdot \Id_C = g h_k \quad \text{for all } k=1,\dots,n.
	\]
	Let $N$ be a finitely presented $R$-module and $s:N\to C$ an $R$-homomorphism. For each $d_k\in J$, we have
	\[
	d_k s = g h_k s = g t_k,
	\]
	where $t_k := h_k s : N\to B$. Consequently, the natural sequence
	\[
	0\to \Hom_R(N,A)\xrightarrow{\Hom_R(N,f)}\Hom_R(N,B)\xrightarrow{\Hom_R(N,g)}\Hom_R(N,C)\to 0
	\]
	is $w$-exact. Therefore, by \cite[Theorem 2.5]{XW18}, $\zeta$ is a $w$-pure short exact sequence.
	
	\medskip
	(2) Let 
	\[
	\{\zeta_i:\ 0\to A_i\xrightarrow{f_i}B_i\xrightarrow{g_i}C_i\to 0\mid i\in\Gamma\}
	\]
	be a direct system of $w$-pure short exact sequences of $R$-modules. For any $R$-module $M$, consider the following commutative diagram:
	\[
	\xymatrix@C=40pt@R=40pt{
		0 \ar[r] & M\otimes_R\varinjlim\limits_{i\in\Gamma}A_i \ar[d]^{\cong} \ar[r]^{M\otimes_R\varinjlim f_i} &
		M\otimes_R\varinjlim\limits_{i\in\Gamma}B_i \ar[d]^{\cong} \ar[r]^{M\otimes_R\varinjlim g_i} &
		M\otimes_R\varinjlim\limits_{i\in\Gamma}C_i \ar[d]^{\cong} \ar[r] & 0 \\
		0 \ar[r] & \varinjlim\limits_{i\in\Gamma}(M\otimes_R A_i) \ar[r]^{\varinjlim(M\otimes f_i)} &
		\varinjlim\limits_{i\in\Gamma}(M\otimes_R B_i) \ar[r]^{\varinjlim(M\otimes g_i)} &
		\varinjlim\limits_{i\in\Gamma}(M\otimes_R C_i) \ar[r] & 0.
	}
	\]
	The vertical maps are natural isomorphisms. Since each $\zeta_i$ is $w$-pure, the sequence $M\otimes_R\zeta_i$ is $w$-exact; hence $\Ker(M\otimes_Rf_i)$ is \GV-torsion for all $i$. Therefore,
	\[
	\Ker\Big(\varinjlim\limits_{i\in\Gamma}(M\otimes_Rf_i)\Big)
	\]
	is \GV-torsion, implying that the bottom row, and hence the top row, is $w$-exact. Consequently, $\varinjlim\limits_{i\in\Gamma}\zeta_i$ is a $w$-pure short exact sequence.
\end{proof}

\medskip

Let $R$ be an integral domain with quotient field $Q$, and let $I$ be a nonzero ideal of $R$.  
We say that $I$ is \emph{$w$-invertible} if
\[
(II^{-1})_w = R,
\]
where $I^{-1} := \{\,x\in Q \mid Ix\subseteq R\,\}$.  
Recall from \cite[Theorem 7.5.9]{wk24} that an integral domain $R$ is a P$v$MD iff every nonzero finitely generated ideal of $R$ is $w$-invertible. Thus, P$v$MDs may be regarded as $w$-analogues (or $v$-analogues) of Pr\"ufer domains.

It is well known that an integral domain $R$ is Pr\"ufer iff the torsion submodule of every $R$-module is pure,  
equivalently, every finitely generated $R$-module decomposes as a direct sum of a torsion module and a projective module,  
or, equivalently, for every finitely generated $R$-module $M$, the canonical exact sequence
\[
0\longrightarrow T\longrightarrow M\longrightarrow M/T\longrightarrow 0,
\]
where $T$ is the torsion submodule of $M$, is split (see \cite[Chapter I, Prop.\ 8.12; Chapter V, Cor.\ 2.9]{FS01}).  
In the next theorem, we establish a $w$-analogue of these classical results.

\begin{theorem}\label{main}
	Let $R$ be an integral domain. Then the following statements are equivalent:
	\begin{enumerate}
		\item $R$ is a P$v$MD.
		\item The torsion submodule of any $($finitely generated$)$ $R$-module is a $w$-pure submodule.
		\item The quotient of a finitely generated $R$-module by its torsion submodule is $w$-projective.
		\item The quotient of a $w$-finitely generated $R$-module by its torsion submodule is $w$-projective.
		\item The quotient of a finitely generated $R$-module by its torsion submodule is $w$-split.
		\item The quotient of a $w$-finitely generated $R$-module by its torsion submodule is $w$-split.
		\item For every finitely generated $R$-module $M$, the canonical exact sequence 
		\[
		0 \longrightarrow T \longrightarrow M \longrightarrow M/T \longrightarrow 0,
		\]
		with $T$ the torsion submodule of $M$, $w$-splits.
		\item For every $w$-finitely generated $R$-module $M$, the canonical exact sequence 
		\[
		0 \longrightarrow T \longrightarrow M \longrightarrow M/T \longrightarrow 0,
		\]
		with $T$ the torsion submodule of $M$, $w$-splits.
	\end{enumerate}
\end{theorem}

\begin{proof}
	\noindent $(1)\Rightarrow(2)$:  
	Let $M$ be an $R$-module, and consider the canonical exact sequence
	\[
	0 \longrightarrow T \longrightarrow M \longrightarrow M/T \longrightarrow 0,
	\]
	where $T$ is the torsion submodule of $M$. Then $M/T$ is torsion-free.  
	By \cite[Theorem 3.5]{wq15}, $M/T$ is $w$-flat, hence $T$ is a $w$-pure submodule of $M$ by \cite[Theorem 2.6]{xw18}.
	
	\smallskip
	$(2)\Rightarrow(1)$:  
	Assume, to the contrary, that $R$ is not a P$v$MD. Then by \cite[Theorem 3.5]{wq15}, we have ${w}$-$\mathrm{w.gl.dim}(R)\geq 2$.  
	Hence there exists a finitely generated ideal $L$ of $R$ which is not $w$-flat by \cite[Proposition 3.3]{wq15}.  
	Thus, there is a finitely generated ideal $I$ of $R$ such that $\Tor^R_1(L,R/I)$ is not \GV-torsion by \cite[Theorem 6.4.4]{wk24}.  
	By Lemma~\ref{cog-tor}, we have $\Hom_R(\Tor^R_1(L,R/I),\tilde{E})\neq 0$,  
	where $\tilde{E}=\Hom_R(R\{x\},E')$ and $E'=\prod_{\m\in w\text{-}\Max(R)}E_R(R\{x\}/\m\{x\})$.  
	Equivalently, $\Ext^1_R(L,\Hom_R(R/I,\tilde{E}))\neq 0$, giving a non-split exact sequence:
	\[
	0 \longrightarrow \Hom_R(R/I,\tilde{E}) \longrightarrow M \longrightarrow L \longrightarrow 0. \tag{a}
	\]
	Since $\Hom_R(R/I,\tilde{E})$ is torsion and $L$ is torsion-free, the assumption $(2)$ implies that $\Hom_R(R/I,\tilde{E})$ is a $w$-pure submodule of $M$.  
	For any finitely presented $R$-module $N$, the induced sequence
	\[
	0 \longrightarrow \Hom_R(N,\Hom_R(R/I,\tilde{E})) \longrightarrow \Hom_R(N,M) \longrightarrow \Hom_R(N,L) \longrightarrow 0 \tag{b}
	\]
	is $w$-exact. Hence $K := \ker(\Hom_R(N,\Hom_R(R/I,\tilde{E})) \to \Hom_R(N,M))$ is \GV-torsion.  
	But $\Hom_R(N,\Hom_R(R/I,\tilde{E}))\cong \Hom_R(N\otimes_RR/I,\tilde{E})$ is \GV-torsion-free since $\tilde{E}$ is \GV-torsion-free.  
	Thus, $K=0$ and sequence (b) is exact, implying that sequence (a) is pure.  
	Since $\tilde{E}$ is pure-injective, so is $\Hom_R(R/I,\tilde{E})$ \cite[Chapter XIII, Lemma 2.1(i)]{FS01},  
	forcing sequence (a) to split, a contradiction.  
	Therefore, $R$ must be a P$v$MD.
	
	\smallskip
	$(1)\Rightarrow(4)$:  
	Let $M$ be a $w$-finitely generated $R$-module, and consider the canonical exact sequence
	\[
	0 \longrightarrow T \longrightarrow M \longrightarrow M/T \longrightarrow 0,
	\]
	where $T$ is the torsion submodule of $M$. Since $M/T$ is torsion-free and $w$-finitely generated,  
	it is $w$-flat by \cite[Theorem 3.5]{wq15}, and hence $w$-projective by Proposition~\ref{wf-wp}.
	
	\smallskip
	$(4)\Rightarrow(3)$, $(6)\Rightarrow(5)$, and $(8)\Rightarrow(7)$: Trivial.
	
	\smallskip
	$(5)\Rightarrow(3)$ and $(6)\Rightarrow(4)$:  
	Since every $w$-split module is $w$-projective \cite[Corollary 2.5]{wq20}, these implications follow.
	
	\smallskip
	$(5)\Rightarrow(7)$ and $(6)\Rightarrow(8)$:  
	Follow from \cite[Proposition 2.4]{wq20}.
	
	\smallskip
	$(3)\Rightarrow(6)$:  
	Let $M$ be a $w$-finitely generated $R$-module. Since $M/T$ is torsion-free and \GV-torsion-free,  
	there exists a finitely generated submodule $N\subseteq M/T$ with $N_w=(M/T)_w$.  
	By assumption, $N$ is $w$-projective, and hence so is $M/T$ \cite[Proposition 6.7.2]{wk24}.  
	By \cite[Exercise 6.38]{wk24}, $M/T$ is $w$-split.
	
	\smallskip
	$(3)\Rightarrow(1)$:  
	Let $M$ be a torsion-free $R$-module. Then $M=\varinjlim_{i\in\Gamma}M_i$, where $\{M_i\}$ is the directed system of finitely generated submodules of $M$ \cite[Lemmas 2.3, 2.5]{GT12}.  
	By assumption, each $M_i$ is $w$-projective, hence $w$-flat \cite[Theorem 6.7.25]{wk24}.  
	Since the class of $w$-flat modules is closed under direct limits, $M$ is $w$-flat.  
	Thus, $R$ is a P$v$MD by \cite[Theorem 3.5]{wq15}.
	
	\smallskip
	$(7)\Rightarrow(2)$:  
	Let $M$ be an $R$-module with torsion submodule $T$.  
	Since $M=\varinjlim_{i\in\Gamma}M_i$ as above, set $T_i=M_i\cap T$. Then $T_i$ is the torsion submodule of $M_i$ and $T=\varinjlim_{i\in\Gamma}T_i$.  
	For each $i$, the inclusion $e_{T_i}:T_i\hookrightarrow M_i$ is $w$-split by assumption, hence $w$-pure by Proposition~\ref{w-sp-w-pure}(1).  
	By Proposition~\ref{w-sp-w-pure}(2), $T=\varinjlim T_i$ is a $w$-pure submodule of $M$.
\end{proof}

Recall that an integral domain $R$ is called a \emph{Krull domain} if every nonzero ideal of $R$ is $w$-invertible.  
Thus, Krull domains can be regarded as $w$-analogues of Dedekind domains.

In \cite{ghl25}, the authors proposed the following open problem:

\medskip
\begin{quote}
	\noindent\textbf{Problem 7.}  
	\emph{Prove or disprove the assertion that every result in classical commutative algebra has a corresponding analogue in $w$-module theory.}
\end{quote}

\medskip
As a motivation, they raised the following question:  
Provide a $w$-version of the classical result that, in a Dedekind domain, every finitely generated module $M$ decomposes as a direct sum of a projective module and a torsion module.  
Equivalently, the quotient $M/T$ of a finitely generated $R$-module $M$ by its torsion submodule $T$ is projective,  
or equivalently, the canonical exact sequence
\[
0 \longrightarrow T \longrightarrow M \longrightarrow M/T \longrightarrow 0
\]
splits for every finitely generated $R$-module $M$.

Since every Krull domain is a P$v$MD, Theorem~\ref{main} immediately yields the following result.

\begin{corollary}
	Let $R$ be a Krull domain. Then the following statements hold:
	\begin{enumerate}
		\item The torsion submodule of any $($finitely generated$)$ $R$-module is a $w$-pure submodule.
		\item The quotient of any $(w$-$)$finitely generated $R$-module by its torsion submodule is $w$-projective $($equivalently, $w$-split$)$.
		\item For every $(w$-$)$finitely generated $R$-module $M$, the canonical exact sequence
		\[
		0 \longrightarrow T \longrightarrow M \longrightarrow M/T \longrightarrow 0,
		\]
		where $T$ is the torsion submodule of $M$, $w$-splits.
	\end{enumerate}
\end{corollary}

\section{$w$-Hattori-Davis Criteria for P$v$MDs}

Having established the $w$-module theoretic properties of Krull domains, 
we now turn to a different direction and investigate unified characterizations 
of P$v$MDs. 

\begin{lemma} \textup{(\cite[Theorem 2]{H57} and \cite[Theorem]{D69})}\label{thm:Prufer-Hattori-Davis}
Let $R$ be an integral domain with quotient field $K$. The following statements are equivalent:
\begin{enumerate}
    \item $R$ is a Pr\"ufer domain.
    \item $\Tor^R_2(M,N)=0$ for every pair of $R$-modules $M$ and $N$.
    \item For every $R$-module $X$ and every torsion-free $R$-module $A$, $\Tor^R_1(X,A)=0$.
    \item For torsion-free $R$-modules $A$ and $B$, the tensor product $A \otimes_R B$ is torsion-free.
    \item For every pair of $K$-algebras $\mathcal T$ and $\mathcal S$, and every pair of
    $R$-subalgebras $A \subseteq \mathcal T$ and $B \subseteq \mathcal S$, the canonical map
    \[
        \mu_{A,B}\colon A\otimes_R B \longrightarrow \mathcal T\otimes_K \mathcal S
    \]
    is a monomorphism.
\end{enumerate}
\end{lemma}

Building on these classical results, we can extend them to the $w$-module context to obtain unified characterizations of P$v$MDs.

\begin{theorem}\label{thm:P$v$MD-Hattori-Davis}
    Let $R$ be an integral domain with quotient field $K$. The following are equivalent:
    \begin{enumerate}
        \item $R$ is a P$v$MD.
        \item $\Tor^R_2(M,N)$ is \GV-torsion for every pair of $R$-modules $M,N$
        \textup{(equivalently, the $w$-weak global dimension of $R$ is $\le 1$)}.
        \item For every $R$-module $X$ and every torsion-free $R$-module $A$, $\Tor^R_1(X,A)$ is \GV-torsion.
        \item For every pair of $K$-algebras $\mathcal T,\mathcal S$ and every pair of
        $R$-subalgebras $A\subseteq\mathcal T$, $B\subseteq\mathcal S$, the canonical map
        \[
        \mu_{A,B}\colon A\otimes_R B \longrightarrow \mathcal T\otimes_K \mathcal S
        \]
        has \GV-torsion kernel \textup{(equivalently, $(\mu_{A,B})_{\fkm}$ is injective for all
            $\fkm\in w\text{-}\Max(R)$)}.
    \end{enumerate}
\end{theorem}

\begin{proof}
	$(1)\Rightarrow(3)$:  
	If $R$ is a P$v$MD, every torsion-free $R$-module $A$ is $w$-flat (\cite[Theorem 7.5.9]{wk24}).  
	By the homological criterion, $A$ is $w$-flat iff $\Tor^R_1(X,A)$ is \GV-torsion for all $X$, hence (3).
	
	\smallskip
	$(3)\Rightarrow(2)$:  
	For any $X,Y$, choose a free precover $F\twoheadrightarrow Y$ with kernel $K$.  
	Since $R$ is a domain, $K\subseteq F$ is torsion-free. By dimension shifting:
	\[
	\Tor^R_2(X,Y)\;\cong\;\Tor^R_1(X,K).
	\]
	By (3), the right-hand side is \GV-torsion, so (2) holds.
	
	\smallskip
	$(2)\Rightarrow(1)$:  
	Assume $\Tor^R_2(-,-)$ is \GV-torsion for all pairs. Let $A$ be torsion-free and $X$ arbitrary.  
	Choose a free precover $F\twoheadrightarrow X$ with kernel $K$. Then:
	\[
	\Tor^R_1(X,A)\;\cong\;\Tor^R_2(K,A),
	\]
	which is \GV-torsion by (2). Hence $A$ is $w$-flat, and thus $R$ is a P$v$MD since $w$-w.gl.dim$(R)\le1$.
	
	\smallskip
	$(1)\Rightarrow(4)$:  
	For any $\fkm\in w\text{-}\Max(R)$, $R_{\fkm}$ is a valuation domain.  
	By Davis’ theorem, the localized map
	\[
	(\mu_{A,B})_{\fkm}\colon A_{\fkm}\otimes_{R_{\fkm}}B_{\fkm} \longrightarrow \mathcal (T\otimes_K\mathcal S)_{\fkm}
	\]
	is injective. Since this holds for all $\fkm$, $\ker(\mu_{A,B})$ is \GV-torsion, so $\mu_{A,B}$ is $w$-injective.
	
	\smallskip
	$(4)\Rightarrow(1)$:  
	Take $\mathcal T=\mathcal S=K$ and $A,B\subseteq K$. By (4), $\mu_{A,B}\colon A\otimes_R B\to K$ has \GV-torsion kernel.  
	Localizing at any $\fkm\in w\text{-}\Max(R)$ gives:
	\[
	A_{\fkm}\otimes_{R_{\fkm}}B_{\fkm} \hookrightarrow K,
	\]
	so $A_{\fkm}\otimes_{R_{\fkm}}B_{\fkm}$ is torsion-free over $R_{\fkm}$.  
	By Davis’ theorem, $R_{\fkm}$ is Pr\"ufer for all $\fkm$, hence $R$ is a P$v$MD.
\end{proof}

\begin{corollary}\label{cor:P$v$MD-Davis-special-corrected}
	Let $R$ be an integral domain with quotient field $K$. The following are equivalent to $R$ being a P$v$MD:
	\begin{enumerate}
		\item For all $R$-subalgebras $A,B\subseteq K$, the map
		\[
		\mu_{A,B}\colon A\otimes_R B \longrightarrow K
		\]
		has \GV-torsion kernel \textup{(equivalently, $(\mu_{A,B})_{\fkm}$ is injective for all $\fkm\in w\text{-}\Max(R)$)}.
		\item For all torsion-free $R$-modules $M,N$ and all $\fkm\in w\text{-}\Max(R)$,
		the localization
		\[
		(M\otimes_R N)_{\fkm} \cong M_{\fkm}\otimes_{R_{\fkm}}N_{\fkm}
		\]
		is torsion-free over the valuation domain $R_{\fkm}$.
	\end{enumerate}
\end{corollary}

\begin{proof}
	$(1)$ This follows from Theorem~\ref{thm:P$v$MD-Hattori-Davis} by taking $\mathcal T=\mathcal S=K$.
	
	$(2)$ 
	If $R$ is a P$v$MD, each $R_{\fkm}$ is a valuation domain, so by Hattori’s theorem,
	$M_{\fkm}\otimes_{R_{\fkm}}N_{\fkm}$ is torsion-free for all torsion-free $M,N$. 
 
	Conversely, if $(2)$ holds, let $M=A$ and $N=B$ be any $R$-subalgebras of $K$.  
	Then $A_{\fkm}\otimes_{R_{\fkm}}B_{\fkm}$ is torsion-free for all $\fkm$, so by Davis’ theorem each $R_{\fkm}$ is Pr\"ufer.  
	Thus, $R$ is a P$v$MD.
\end{proof}

In Hattori's original theorem \cite{H57}, the first four conditions listed in Lemma \ref{thm:Prufer-Hattori-Davis} are mutually equivalent, providing a characterization of Pr\"ufer domains. However, the following example shows that in the $w$-setting, the $w$-analogue of the fourth condition is not equivalent to the definition of a P$v$MD; that is, the tensor product of two ($\GV$-)torsion-free modules may fail to be $\GV$-torsion-free over P$v$MDs.

\begin{example}
	Let $R=k[x,y]$, where $k$ is a field, and put $M=N=\langle x,y\rangle_R$ (the ideal generated by $x$ and $y$).
	Then $R$ is a P$v$MD and $M,N$ are torsion-free.
	However, $M\otimes_R N$ is \emph{not} \GV-torsion-free.
	
	Indeed, consider the element
	\[
	z := x\otimes y - y\otimes x \ \in\ M\otimes_R N .
	\]
	We compute, using the identity $r\cdot(m\otimes n) = (rm)\otimes n = m\otimes(rn)$:
	\begin{align*}
		x\cdot z
		&= x(x\otimes y - y\otimes x)
		= x^2\otimes y - xy\otimes x
		= x\otimes xy - x\otimes yx
		= x\otimes(xy - yx) = 0, \\[2mm]
		y\cdot z
		&= y(x\otimes y - y\otimes x)
		= yx\otimes y - y^2\otimes x
		= y\otimes xy - y\otimes yx
		= y\otimes(xy - yx) = 0.
	\end{align*}
	Hence, $J\cdot z=0$ where $J=\langle x,y\rangle$.
	Since $J\in \GV(R)$ and $z\neq 0$, it follows that $M\otimes_R N$ has a nonzero \GV-torsion element.
	Thus, $M\otimes_R N$ is \emph{not} \GV-torsion-free.
\end{example}

\section{$w$-Flat Epimorphism Criteria for Overrings of P$v$MDs}

In order to further investigate structural properties of P$v$MDs, 
we now recall the notions of ring epimorphisms and their $w$-analogues, 
which will be essential for the subsequent characterization results.

\medskip

A ring homomorphism $\varphi:R\to S$ is called a \emph{$($categorical$)$ epimorphism of rings}
if for every pair of ring maps $\alpha,\beta:S\to T$, the equality
$\alpha\circ\varphi=\beta\circ\varphi$ implies $\alpha=\beta$.
Equivalently, $\varphi$ is a ring epimorphism iff the
multiplication map
\[
\mu:\ S\otimes_R S \longrightarrow S,\qquad s\otimes t\mapsto st,
\]
is an isomorphism of rings (\cite[Theorem 1.2.19]{G89}).

\medskip

A ring homomorphism $\varphi:R\to S$ is a \emph{flat epimorphism} if:
\begin{enumerate}
	\item $S$ is a flat $R$-module; and
	\item $\varphi$ is a ring epimorphism.
\end{enumerate}

\begin{definition}\label{def:w-epi}
	Let $R$ be a domain and $\varphi:R\to T$ a ring homomorphism.  
	We say that $\varphi$ is a \emph{$w$-epimorphism of rings} if the canonical multiplication map
	\[
	\mu:\ T\otimes_R T \longrightarrow T
	\]
	has \GV-torsion kernel and cokernel.  
	Equivalently, for every $\fkm\in w\textrm{-}\Max(R)$, the localized map 
	\[
	\mu_{\fkm}:(T\otimes_R T)_{\fkm}\to T_{\fkm}
	\]
	is an isomorphism, i.e.\ $\varphi_{\fkm}:R_{\fkm}\to T_{\fkm}$ is a ring epimorphism.
\end{definition}

\begin{definition}\label{def:w-flat-epi}
	A ring map $\varphi:R\to T$ is a \emph{$w$-flat epimorphism} if:
	\begin{enumerate}
		\item $T$ is $w$-flat as an $R$-module; and
		\item $\varphi$ is a $w$-epimorphism.
	\end{enumerate}
	Equivalently, for every $\fkm\in w\textrm{-}\Max(R)$, the localization
	\[
	\varphi_{\fkm}:R_{\fkm}\to T_{\fkm}
	\]
	is a \emph{flat epimorphism} of rings.
\end{definition}

\begin{remark}
	When $R$ is a P$v$MD, the localizations $R_{\fkm}$ with $\fkm\in w\textrm{-}\Max(R)$ are valuation (hence Pr\"ufer) domains.
	Thus, the $w$-epimorphism condition reduces to checking flat epimorphism after these localizations.
	In particular, localizations $R\to S^{-1}R$ are flat epimorphisms, and the $w$-version asserts that this holds $w$-locally.
\end{remark}

\medskip

The next result is the $w$-analogue of \cite[Proposition]{S69}:  An integer domain is a Pr\"ufer domain if and only if $R\hookrightarrow T$ is a flat epimorphism for every overring $T$ with $R\subseteq T\subseteq K$.  
After $w$-localizing at each $\fkm\in w\text{-}\Max(R)$, the classical characterization  
``every overring is a flat epimorphism" extends from Pr\"ufer domains to P$v$MDs.

\begin{theorem}\label{thm:pvmd-flat-epi}
	Let $R$ be a domain with quotient field $K$. The following are equivalent:
	\begin{enumerate}
		\item $R$ is a P$v$MD.
		\item For every overring $T$ with $R\subseteq T\subseteq K$ and every $\fkm\in w\text{-}\Max(R)$,
		the localized inclusion
		\[
		\iota_{\fkm}\colon R_{\fkm}\ \longrightarrow\ T_{\fkm}
		\]
		is a \emph{flat ring epimorphism}. Equivalently, the multiplication map
		\[
		\mu_{\fkm}\colon T_{\fkm}\otimes_{R_{\fkm}}T_{\fkm}\ \xrightarrow{\ \cong\ }\ T_{\fkm}
		\]
		is an isomorphism and $T_{\fkm}$ is flat over $R_{\fkm}$.
	\end{enumerate}
	In particular, under \textup{(2)}, each overring $T$ is $w$-flat as an $R$-module, and the global multiplication map
	\[
	\mu:T\otimes_R T\to T
	\]
	has \GV-torsion kernel and cokernel, i.e.\ $\iota:R\hookrightarrow T$ is a $w$-epimorphism of rings.
\end{theorem}

\begin{proof}
	(1)$\Rightarrow$(2):  
	Fix $\fkm\in w\text{-}\Max(R)$. Since $R$ is a P$v$MD, $R_{\fkm}$ is a valuation domain.  
	For valuation domains, any overring $U$ with $V\subseteq U\subseteq qf(V)$ is a localization of $V$,  
	hence $V\to U$ is a flat ring epimorphism.  
	Applying this to $V=R_{\fkm}$ and $U=T_{\fkm}$ gives the claim.  
	By \cite[Theorem 1.2.19]{G89}, $\mu_{\fkm}$ is an isomorphism, and flatness follows since $T_{\fkm}$ is a localization.  
	Thus, $\iota$ is a $w$-flat epimorphism.
	
	\smallskip
	(2)$\Rightarrow$(1): 
	Fix $\fkm\in w\text{-}\Max(R)$ and let $U$ be any overring of $R_{\fkm}$ in $K$.  
	Then $U=T_{\fkm}$ for the same ring $T:=U$ viewed as an overring of $R$; by hypothesis,  
	$R_{\fkm}\to U$ is a flat ring epimorphism.  
	A classical characterization states that a domain $A$ is Pr\"ufer iff for every overring $B$ of $A$ in $qf(A)$,  
	the inclusion $A\subseteq B$ is a flat epimorphism.  
	Thus, $R_{\fkm}$ is Pr\"ufer for all $\fkm\in w\text{-}\Max(R)$ \cite[Proposition]{S69},  
	and hence $R$ is a P$v$MD.
\end{proof}

\section{Characterizing P$v$MDs via Pure $w$-Injective Modules} 

In this section, we extend a classical result (\cite[Theorem 5]{FM91}) of Fuchs and Meijer, which states that an integral domain $R$ is Pr\"ufer if and only if every pure injective divisible $R$-module is injective, to the setting of P$v$MDs. 
By investigating pure $w$-injective modules, we provide a homological criterion for P$v$MDs, showing that an integral domain is a P$v$MD if and only if every pure $w$-injective divisible module is injective.

Recall from \cite{zkzh25} that an $R$-module $E$ is called \emph{$w$-universal injective} if $E$ is an injective $R$-module and, for any $\GV$-torsion-free $R$-module $A$ and any nonzero element $x \in A$, there exists an $R$-homomorphism $f : A \rightarrow E$ such that $f(x) \neq 0$.

We always set $\widehat{E} = \prod\limits_{\m \in w\text{-}\Max(R)} E(R/\m)$. It follows from \cite[Lemma 2.2, Proposition 2.3]{zkzh25} that $\widehat{E}$ is $\GV$-torsion-free and $w$-universal injective.

\begin{lemma}\label{lem:w-pure-section}
	A $w$-monomorphism $N \rightarrow M$ of $R$-modules is $w$-pure if and only if the natural homomorphism
	\[
	\Hom_{R}(M, \widehat{E}) \twoheadrightarrow \Hom_{R}(N, \widehat{E})
	\]
	admits a section.
\end{lemma}

\begin{proof}
	Suppose the $w$-monomorphism $f: N \rightarrow M$ is $w$-pure. Then, for any $R$-module $L$, the sequence 
	\[
	0 \rightarrow N \otimes_R L \xrightarrow{f \otimes_R L} M \otimes_R L
	\]
	is $w$-exact. Since $\widehat{E}$ is $\GV$-torsion-free and injective, the induced sequence
	\[
	\Hom_R(M \otimes_R L, \widehat{E}) \rightarrow \Hom_R(N \otimes_R L, \widehat{E}) \rightarrow 0
	\]
	is exact. Consequently, the sequence
	\[
	\Hom_R(L, \Hom_R(M, \widehat{E})) \rightarrow \Hom_R(L, \Hom_R(N, \widehat{E})) \rightarrow 0
	\]
	is exact for any $R$-module $L$. Taking $L = \Hom_R(N, \widehat{E})$, we obtain that the natural homomorphism 
	\[
	\Hom_R(M, \widehat{E}) \twoheadrightarrow \Hom_R(N, \widehat{E})
	\]
	admits a section.
	
	Conversely, assume that the natural homomorphism 
	\[
	\Hom_R(M, \widehat{E}) \twoheadrightarrow \Hom_R(N, \widehat{E})
	\]
	admits a section. Then, for any $R$-module $L$, the sequence
	\[
	\Hom_R(L, \Hom_R(M, \widehat{E})) \rightarrow \Hom_R(L, \Hom_R(N, \widehat{E})) \rightarrow 0
	\]
	is exact. Hence,
	\[
	\Hom_R(M \otimes_R L, \widehat{E}) \rightarrow \Hom_R(N \otimes_R L, \widehat{E}) \rightarrow 0
	\]
	is exact. Since $\widehat{E}$ is $w$-universal injective, it follows from \cite[Theorem 2.5]{zkzh25} that 
	\[
	0 \rightarrow N \otimes_R L \xrightarrow{f \otimes_R L} M \otimes_R L
	\]
	is $w$-exact for any $R$-module $L$. Consequently, the $w$-monomorphism $N \rightarrow M$ is $w$-pure.
\end{proof}

\begin{lemma}\label{w-pure-mono}
	Let $M$ be an $R$-module. Then the natural morphism 
	\[
	\phi_M: M \longrightarrow \Hom_{R}(\Hom_{R}(M, \widehat{E}), \widehat{E})
	\]
	is a $w$-pure $w$-monomorphism.
\end{lemma}

\begin{proof}
	First, let $T$ denote the $\GV$-torsion submodule of $M$. Consider the following commutative diagram with exact rows, where $(-,-) := \Hom_R(-,-)$:
	\[
	\xymatrix@C=30pt@R=30pt{
		0 \ar[r] & T \ar[d]^{\phi_T} \ar[r] & M \ar[d]^{\phi_M} \ar[r]^{\pi} & M/T \ar[d]^{\phi_{M/T}} \ar[r] & 0 \\
		0 \ar[r] & ((T, \widehat{E}), \widehat{E}) \ar[r] & ((M, \widehat{E}), \widehat{E}) \ar[r]^{\cong} & ((M/T, \widehat{E}), \widehat{E}) \ar[r] & 0
	}
	\]
	It follows from \cite[Theorem 2.5]{zkzh25} that $\phi_{M/T}$ is a monomorphism. Hence, $\phi_M$ is a $w$-monomorphism.
	
	Next, we show that $\phi_M$ is $w$-pure. By Lemma~\ref{lem:w-pure-section}, it suffices to show that $\Hom_R(\phi_M, \widehat{E})$ admits a section. Indeed, the map
	\[
	\phi_{\Hom_R(M, \widehat{E})}: \Hom_R(M, \widehat{E}) \longrightarrow \Hom_R(\Hom_R(\Hom_R(M, \widehat{E}), \widehat{E}), \widehat{E})
	\]
	provides such a section.
\end{proof}

\begin{definition}
	An $R$-module $E$ is said to be \emph{pure $w$-injective} if it is injective with respect to all $w$-pure $w$-monomorphisms $f: N \rightarrow M$.
\end{definition}

It is easy to check that the class of pure $w$-injective modules is closed under direct products and direct summands.

\begin{proposition}\label{pw-injchar}
	A $\GV$-torsion-free $R$-module $M$ is pure $w$-injective if and only if it is a direct summand of an $R$-module of the form $\Hom_R(L, \widehat{E})$ for some $R$-module $L$.
\end{proposition}

\begin{proof}
	Let $M$ be a $\GV$-torsion-free pure $w$-injective $R$-module. By Lemma~\ref{w-pure-mono}, the natural morphism
	\[
	\phi_M: M \longrightarrow \Hom_R(\Hom_R(M, \widehat{E}), \widehat{E})
	\]
	is a $w$-pure $w$-monomorphism. Since $M$ is pure $w$-injective, this map splits, and hence $M$ is a direct summand of $\Hom_R(\Hom_R(M, \widehat{E}), \widehat{E})$.
	
	Conversely, we show that any module of the form $\Hom_R(L, \widehat{E})$ is pure $w$-injective and $\GV$-torsion-free. Let $f: N \rightarrowtail M$ be a $w$-pure monomorphism and let $L$ be an $R$-module. Then the sequence
	\[
	0 \rightarrow N \otimes_R L \xrightarrow{f \otimes_R L} M \otimes_R L
	\]
	is $w$-exact. Since $\widehat{E}$ is $\GV$-torsion-free and injective, the induced sequence
	\[
	\Hom_R(M \otimes_R L, \widehat{E}) \longrightarrow \Hom_R(N \otimes_R L, \widehat{E}) \longrightarrow 0
	\]
	is exact. Equivalently, 
	\[
	\Hom_R(M, \Hom_R(L, \widehat{E})) \longrightarrow \Hom_R(N, \Hom_R(L, \widehat{E})) \longrightarrow 0
	\]
	is exact. Consequently, $\Hom_R(L, \widehat{E})$ is pure $w$-injective and, trivially, $\GV$-torsion-free.
\end{proof}

\begin{theorem} \label{thm:pvmd-winj}
	Let $R$ be an integral domain. Then $R$ is a P$v$MD if and only if every pure $w$-injective divisible $R$-module is injective.
\end{theorem}

\begin{proof}
	Suppose first that $R$ is a P$v$MD. Let $M$ be a pure $w$-injective divisible $R$-module. By \cite[Theorem 2.10]{XW18}, $M$ is absolutely $w$-pure. Consider the natural embedding 
	\[
	i: M \hookrightarrow E(M),
	\] 
	where $E(M)$ is the injective envelope of $M$. It follows from \cite[Theorem 2.6]{XW18} that $i$ is a $w$-pure monomorphism, and hence splits since $M$ is pure $w$-injective. Therefore, $M$ is injective.
	
	Conversely, assume that every pure $w$-injective divisible module is injective. Let $M$ be a divisible $R$-module, and consider an exact sequence 
	\[
	0 \longrightarrow M \xrightarrow{f} A \xrightarrow{g} B \longrightarrow 0
	\]
	of $R$-modules. Consider the following commutative diagram of exact sequences, where $(-,-):=\Hom_R(-,-)$:
	\[
	\xymatrix@C=30pt@R=30pt{
		0 \ar[r] & M \ar[d]^{\phi_M} \ar[r]^{f} & A \ar[d]^{a} \ar[r]^{g} & B \ar@{=}[d] \ar[r] & 0 \\
		0 \ar[r] & ((M, \widehat{E}), \widehat{E}) \ar[r]^{b} & C \ar[r] & B \ar[r] & 0 
	}
	\]
	
	By Lemma~\ref{w-pure-mono}, $\phi_M$ is a $w$-monomorphism. Note that $\Hom_R(\Hom_R(M, \widehat{E}), \widehat{E})$ is divisible since $M$ is divisible. By Proposition~\ref{pw-injchar}, it is also pure $w$-injective, and therefore injective by our assumption. Consequently, the bottom row of the diagram splits.
	
	It follows from Lemma~\ref{w-pure-mono} that $\phi_M$ is a $w$-pure $w$-monomorphism. By the proof of \cite[Proposition 2.2(1)]{L22}, the composition $b \circ \phi_M$ is a $w$-pure $w$-monomorphism, and hence so is $f$ by \cite[Proposition 2.2(2)]{L22}. Therefore, $M$ is absolutely $w$-pure. By \cite[Theorem 2.10]{XW18}, this implies that $R$ is a P$v$MD.
\end{proof}


\bigskip

\noindent {\bf Fund:}
H. Kim was supported by Basic Science Research Program through the National Research Foundation of Korea(NRF) funded by the Ministry of Education (2021R1I1A3047469).

\end{document}